\documentclass[a4paper, 11pt]{article}

\usepackage{amsmath}
\usepackage{amssymb}
\usepackage{amsthm}
\newtheorem{thm}{Theorem}[section]
\newtheorem{cor}[thm]{Corollary}
\newtheorem{lem}[thm]{Lemma}

\theoremstyle{definition}

\numberwithin{equation}{section}
\theoremstyle{remark}
\newtheorem{rem}{Remark}[section]
\date{}

\newcommand{\be}{\begin{eqnarray}}
\newcommand{\en}{\end{eqnarray}}
\newcommand{\no}{\nonumber}

\newcommand{\ov}{\overline}

\newcommand{\si}{\sigma}

\newcommand{\fr}{\frac}
\newcommand{\pa}{\partial}

\newcommand{\D}{\Delta}
\newcommand{\om}{\Omega}
\newcommand{\na}{\nabla}
\newcommand{\lan}{\langle}
\newcommand{\ran}{\rangle}
\newcommand{\lam}{\lambda}

\newcommand{\ri}{\rightarrow}

\newcommand{\vs}{\vskip0.3cm}
\newcommand{\um}{\mathbf u}
\newcommand{\into}{\int_{\Omega}}

\title{ On  eigenvalues of a system of elliptic equations and  of the biharmonic operator}

\author{Daguang Chen*, Qing-Ming Cheng**, Qiaoling Wang$^{\dag}$, Changyu Xia$^{\ddag}$}

\begin{document}
\maketitle

\begin{abstract}
Let $\om $ be a  bounded domain in an $n$-dimensional Euclidean space $\Bbb R^n$. We study eigenvalues of an eigenvalue problem of a system of  elliptic equations:
 $$
\left \{ \aligned &\Delta {\mathbf u}+ \alpha{\rm   grad}(\text{div}{\mathbf u})=-\sigma {\mathbf u},  \ \text{in $\Omega$}, \\
 &{\mathbf u}|_{\partial \Omega}={\mathbf 0}.
\endaligned \right.
$$
Estimates for eigenvalues of  the above eigenvalue problem are obtained.
Furthermore,  we obtain an  upper bound on  the $(k+1)^{\text{th}}$ eigenvalue $\sigma_{k+1}$.
We also obtain sharp lower bound for the first eigenvalue of two kinds  of  eigenvalue problems
of the biharmonic operator on compact manifolds with boundary and positive Ricci curvature.
\end{abstract}

\footnotetext{2000{\it Mathematics Subject Classification}: 35P15, 53C20, 53C42, 58G25}

\footnotetext{*Research partially supported by China Postdoctoral Science Foundation (No.20080430351).

** Research partially supported by a Grant-in-Aid for Scientific Research from JSPS.

$^\dag$ Research partially supported by CNPq, CAPES/PROCAD.

$^\ddag$ Research partially supported by CNPq, CAPES/PROCAD.}

\renewcommand{\sectionmark}[1]{}
\section{Introduction}
Let $\om$ be a  bounded domain with smooth boundary in an $n$-dimensional Euclidean space $\Bbb R^n$.
Consider an eigenvalue problem of a system of $n$ elliptic equations:
\be  \label{EP1}
\left\{\begin{array}{cl}
\D{\bf u}  +\alpha\ {\rm grad} ({\text{div}\ }{\bf  u})= -\si {\bf u},  \ \ \ {\rm in} \ \ \om,
\\
{\bf u}|_{\pa \om}={\bf 0},
\end{array}\right.
\en
where $\D$ is the Laplacian in $\Bbb R^n$, ${\bf u}=(u_1, u_2, \cdots, u_n)$ is a vector-valued function from $\om $ to $\Bbb R^n$, $\alpha$ is a
non-negative constant, ${\rm div\ }{\bf u}$ denotes the divergence of ${\bf u}$ and ${\rm grad} f$ is the gradient of a function $f$.
Let
\be \no 0<\si_1\leq \si_2\leq\cdots \leq \si_k \leq \cdots \ri \infty
\en
be the eigenvalues of the   problem (\ref{EP1}).
Here each eigenvalue is repeated according to its multiplicity. When $n=3$, the problem (\ref{EP1}) describes the behavior of the elastic vibration \cite{P39}.

Interesting estimates for the eigenvalues of (\ref{EP1}) have been done during the past years. In 1985, Levine and Protter \cite{LevPro85} proved
\be \label{LP85}
\sum_{i=1}^k \si_i \geq \fr{4\pi^2 n}{n+2} \fr{k^{1+2/n}}{(V\omega_{n-1})^{2/n}}\ {\rm  for\ } k=1, 2, \cdots,
\en
where $\omega_{n-1}$ is the volume of the $(n-1)$-dimensional unit sphere.  Furthermore, Hook \cite{Hook90} has studied universal inequalities for eigenvalues of
(\ref{EP1}) and proved
\be \label{Hook}
\sum_{i=1}^k \fr{\si_i}{\si_{k+1}-\si_i}\geq \fr{n^2k}{4(n+\alpha)}, \ {\rm for\ } k=1, 2, \cdots.
\en
The method given by Hook to prove the above inequality is abstract.
Levitin and Parnovski \cite{LPar02} have obtained
\be \label{LP92}
\si_{k+1}-\si_k \leq \fr{\max\{4+\alpha^2; (n+2)\alpha+8\}}{n+\alpha} \fr 1k \sum_{i=1}^k \si_i, \ {\rm for\ } k=1, 2, \cdots.
\en
Recently,  by  making use of a direct and explicit method, Cheng and Yang \cite{CY09} have proved the following   universal inequality of Yang type :
 \be \label{CY}
 \sum_{i=1}^k (\si_{k+1}-\si_i)\leq \fr{2\sqrt{n+\alpha}}n\left\{\sum_{i=1}^k (\si_{k+1}-\si_i)^{\fr 12}\sum_{i=1}^k (\si_{k+1}-\si_i)^{\fr 12}\si_i\right\}^{\fr 12}.
\en

In this paper, we  strengthen  the above Cheng-Yang's inequality.
\begin{thm}\label{Mthm}
 Let $\Omega$ be a  bounded domain
 in an $n$-dimensional Euclidean space $\Bbb R^n$.  Eigenvalues of the eigenvalue problem {\rm (1.1)}
satisfy
\begin{equation}\label{mineq}
\aligned
\sum_{i=1}^{k}(\sigma_{k+1}-\sigma_i)^2
\leq\min\biggl\{\dfrac{4(n+\alpha)}{n^2}, \dfrac{A(n,\alpha)}{n+\alpha}\biggl\}\sum_{i=1}^k(\sigma_{k+1}-\sigma_i)\sigma_i,
\endaligned
\end{equation}
where $A(n,\alpha)$ is defined by
$$
A(n,\alpha)=\begin{cases} 4+\alpha^2,
& \text{if}\  \  \alpha\geq \dfrac{n+2+\sqrt{(n+2)^2+16}}2,\\
\dfrac{8+(n+2)\alpha}{1+L}, & \text{ if}\  \ 0\leq \alpha< \dfrac{n+2+\sqrt{(n+2)^2+16}}2,
\end{cases}
$$
with $L=\dfrac{\{4+(n+2)\alpha-\alpha^2\}n^2}{4(n+\alpha)^2}>0$,
 $\sigma_i$ denotes the
$i^{\text{th}}$  eigenvalue  of {\rm (1.1)}.
\end{thm}

\begin{cor}\label{Y1}
Under the same assumptions as in the theorem 1.1, we have
\begin{equation}\label{Y1}
\sum_{i=1}^{k}(\sigma_{k+1}-\sigma_i)^2\leq\frac{4(n+\alpha)}{n^2}
\sum_{i=1}^k(\sigma_{k+1}-\sigma_i)\sigma_i.
\end{equation}
\end{cor}
We will show in the next section that (\ref{Y1}) implies (\ref{CY}).

\begin{rem}
For  $\alpha=0$, our result becomes the sharper inequality of
Yang {\rm \cite {Yang91}} ( Cf. \cite{CY07}). Universal inequalities of Payne-P\'olya-Weinberger-Yang type for eigenvalues of elliptic operators on Riemannian manifolds have been studied recently by many mathematicians. One can find various interesting results in this direction, e.g., in \cite{Ash99}-\cite{Ash04}, \cite{CC08}-\cite{CY07},  \cite{HSI09}-\cite{PPW56}, \cite{WX071}-\cite{Yang91}, etc.
\end{rem}

The inequality (\ref{mineq}) is a quadratic inequality of $\si_{k+1}$. By solving it, one can get an explicit upper bound on $\si_{k+1}$ in terms of $\si_1, \cdots, \si_k$.
\begin{cor}
From the theorem \ref{Mthm}, it is  not hard to obtain the following simple  inequality
$$
\aligned
\sigma_{k+1}\leq\biggl(1+\min\biggl\{\dfrac{4(n+\alpha)}{n^2}, \dfrac{A(n,\alpha)}{n+\alpha}\biggl\}\biggl)\dfrac1k\sum_{i=1}^k\sigma_i.
\endaligned
$$
and the gap of any consecutive eigenvalues
$$
\aligned
\sigma_{k+1}-\sigma_k\leq\min\biggl\{\dfrac{4(n+\alpha)}{n^2}, \dfrac{A(n,\alpha)}{n+\alpha}\biggl\}\dfrac1k\sum_{i=1}^k\sigma_i.
\endaligned
$$
\end{cor}
For lower order eigenvalues of the eigenvalue problem (\ref{EP1}), Yang and the second
author \cite{CY09} proved the following
\begin{equation}\label{CYL}
\sigma_2+\sigma_3+\cdots + \sigma_{n+1}\leq n\sigma_1+4(1+\alpha)
\sigma_1.
\end{equation}
Combining Theorem \ref{Mthm} and (\ref{CYL}), we can derive an upper bound for eigenvalue $\sigma_{k+1}$.

\begin{cor}
Under the same assumptions  as in Theorem \ref{Mthm}, we have
 $$
\sigma_{k+1}
\leq \Big(1+\frac {a(n)(n+\alpha)}{n^2}\Big)k^{\frac{2(n+\alpha)}{n^2}}\sigma_1,
 $$
where $a(n)\leq 4$ can be explicitly given.
\end{cor}
\begin{proof}
From ( \ref{Y1})  and (\ref{CYL}), our result is proved by  applying
the recursion formula of Cheng and Yang \cite{CY07} to our case.
We need to  notice that  the recursion formula
of Cheng and Yang \cite{CY07} does hold  for any positive real number $n$.
In our case, it is $\dfrac{n^2}{n+\alpha}$.
\end{proof}

The classical Lichnerowicz-Obata theorem states that if $M$ is
an n-dimensional complete connected Riemannian manifold with Ricci curvature
bounded below by $(n-1)$ then the first non-zero eigenvalue of the Laplacian of $M$ is bigger than or equal to $n$  with equality holding if and only if $M$ is isometric to a unit $n$-sphere (Cf. \cite{Chav84}).
In 1977, Reilly obtained a similar result for the first  Dirichlet eigenvalue of the Laplacian of compact manifold with boundary. Reilly's theorem can be stated as follows. Let $M$ be an $n(\geq 2)$-dimensional compact connected Riemannian manifold with boundary $\pa M$. Assume that the  Ricci curvature of $M$ is bounded below by $(n-1)$. If the mean curvature of $\pa M$ is non-negative, then the first  Dirichlet eigenvalue $\lambda_1$ of the Laplacian of
$M$  satisfies $\lam_1\geq n$ with equality holding if and only $M$ is isometric to an $n$-dimensional  Euclidean unit semi-sphere (Cf. \cite{Reilly77}). A similar estimate for the first non-zero Neumann eigenvalue of the Laplacian of the same manifolds has been obtained in \cite{Esc90} and \cite{Xia91} independently.

 The second part of this paper is to  estimate lower bounds  for the first eigenvalue of  four kinds of the
 eigenvalue problems of the biharmonic operator on compact manifolds with boundary and positive Ricci curvature. The first two results in this direction concerns the clamped and the buckling problem.

 \begin{thm}\label{EPbi1}
 Let $(M, \lan, \ran )$ be an $n(\geq 2)$-dimensional compact connected Riemannian manifold with boundary $\pa M$ and denote by $\nu$ the outward unit normal vector field of $\pa M$. Assume that the Ricci curvature of $M$ is bounded below by $(n-1)$.  Let $\lambda_1$ be the first eigenvalue with  Dirichlet boundary condition of the Laplacian of  $M$ and let $\Gamma_1$ be the first eigenvalue of the clamped plate problem on $M$:
\be \label{EP2}
\left\{\begin{array}{l}
  \Delta^2 u= \Gamma u   \ \ {\rm in \ \ } M, \\
 u= \fr{\pa u}{\pa \nu}=0 \ \ {\rm on \ \ } \pa M.
\end{array}\right.
\en
Then we have $\Gamma_1> n\lambda_1$.
\end{thm}

\begin{thm}\label{EPbu1}
 Assume  $M$ satisfy the conditions  in Theorem 1.5  and let $\Lambda_1$ be the first eigenvalue of the following
 buckling problem:
\be \label{EP3}
\left\{\begin{array}{l}
  \Delta^2 u= -\Lambda \Delta u   \ \ {\rm in \ \ } M, \\
 u= \fr{\pa u}{\pa \nu}=0 \ \ {\rm on \ \ } \pa M.
\end{array}\right.
\en
Then $\Lambda_1> n$.
\end{thm}

We then consider two different eigenvalue problems of the biharmonic operator and obtain sharp lower bound for the first eigenvalues of them.

 \begin{thm}\label{EPbi2}
 Let $(M, \lan, \ran )$ be an $n(\geq 2)$-dimensional compact connected Riemannian manifold with boundary $\pa M$. Assume that the Ricci curvature of $M$ is bounded below by $(n-1)$ and  that the mean curvature of $\pa M$ is non-negative.  Let $\lambda_1$ be the first  eigenvalue with  Dirichlet boundary condition of  the Laplacian of $M$ and let $p_1$ be the first eigenvalue of the following
 problem :
\be \label{EP4}
\left\{\begin{array}{l}
  \Delta^2 u= p  u   \ \ {\rm in \ \ } M, \\
 u= \fr{\pa^2 u}{\pa \nu^2}=0 \ \ {\rm on \ \ } \pa M.
\end{array}\right.
\en
Then $p_1\geq n\lambda_1$ with equality holding if and only if $M $ is isometric to an $n$-dimensional Euclidean unit semi-sphere.
\end{thm}
\begin{thm}\label{EPbu2}
Assume  $M$ satisfy the conditions  in Theorem 1.7 and let $q_1$ be the first eigenvalue of the following
 problem :
\be \label{EP5}
\left\{\begin{array}{l}
  \Delta^2 u= -q \Delta u   \ \ {\rm in \ \ } M, \\
 u= \fr{\pa^2 u}{\pa \nu^2}=0 \ \ {\rm on \ \ } \pa M.
\end{array}\right.
\en
Then $q_1\geq n$ with equality holding if and only if $M$ is isometric to an $n$-dimensional Euclidean unit semi-sphere.
\end{thm}

\section{A Proof of Theorem \ref{Mthm}}

In this section, we shall prove the inequality (\ref{mineq}) and derive that this inequality implies Cheng-Yang's inequality (\ref{CY}).
Firstly, we give
some general estimates for eigenvalues of the  problem (\ref{EP1}).
\begin{lem}\label{Mlem}
Let $\Omega$ be a  bounded domain
 in an $n$-dimensional Euclidean space $\Bbb R^n$. Let  $\sigma_i$ denote the $i^{\text{th}}$ eigenvalue
of the eigenvalue problem (\ref{EP1}) and ${\mathbf u}_i$ be the orthonormal vector-valued eigenfunction corresponding
to $\sigma_i$.  For  any function $f\in C^2(\Omega)\cap C^1(\bar{\Omega})$,
we have
$$
\aligned
&\sum_{i=1}^{k}(\sigma_{k+1}-\sigma_i)^2\biggl \{\int_{\Omega}|{\rm   grad}f|^2|{\mathbf u}_i|^2
+\alpha\int_{\Omega}|{\rm   grad}f\cdot{\mathbf u}_i|^2\biggl\}\\
&\leq\sum_{i=1}^k(\sigma_{k+1}-\sigma_i)
\bigl\|2{\rm   grad}f\cdot{\rm   grad}({\mathbf u}_i)+\Delta f {\mathbf u}_i
+\alpha\bigl\{{\rm grad}({\rm   grad}f\cdot{\mathbf u}_i)+{\rm div}({\mathbf u}_i){\rm grad}f\bigl\}\bigl\|^2
\endaligned
$$
and,  for any positive constant $B$,
\begin{equation}\label{Eq1}
\aligned
&\sum_{i=1}^{k}(\sigma_{k+1}-\sigma_i)^2\biggl \{(1-B)\int|{\rm   grad}f|^2|{\mathbf u}_i|^2
-B\alpha\int|{\rm   grad}f\cdot{\mathbf u}_i|^2\biggl\}\\
&\leq\dfrac1B\sum_{i=1}^k(\sigma_{k+1}-\sigma_i)\bigl\|{\rm grad}f\cdot{\rm   grad}({\mathbf u}_i)
+\dfrac12\Delta f {\mathbf u}_i\bigl\|^2,
\endaligned
\end{equation}
where ${\rm grad} f\cdot{\rm   grad} ({\mathbf u}_i)$ is  defined by
$$
{\rm grad}f\cdot{\rm   grad}({\mathbf u}_i)=({\rm   grad}f\cdot{\rm   grad} (u_i^1),
{\rm grad}f\cdot{\rm   grad} (u_i^2), \cdots, {\rm   grad}f\cdot{\rm   grad} (u_i^n)).
$$
\end{lem}

\begin{proof}
Since  ${\mathbf u}_i$ is  the orthonormal vector-valued eigenfunction
corresponding to the $i^{\text{th}}$ eigenvalue $\sigma_i$,
$ {\mathbf u}_i$  satisfies
\begin{equation}\label{Eq2}
\left \{ \aligned &\Delta {\mathbf u}_i+\alpha{\rm   grad}(\text{div}({\mathbf u}_i))=-\sigma_i {\mathbf u}_i,
\quad \text{ in}\,  \Omega ,\\
 &{\mathbf u} _i|_{\partial \Omega}= {\mathbf 0}, \\
&\int_{\Omega} {\mathbf u}_i\cdot {\mathbf u}_j =\delta_{ij}, \ \text{for any $i, j$}.
\endaligned
\right.
\end{equation}
Defining   vector-valued functions ${\mathbf v}_i$ by
\begin{equation}\label{Eq3}
 {\mathbf v}_i =f{\mathbf u}_i -\sum_{j=1}^k a_{ij}{\mathbf u}_j,
\end{equation}
where $a_{ij}=\int_{\Omega} f{\mathbf u}_i\cdot {\mathbf u}_j=a_{ji}$,  we have
\begin{equation}\label{Eq4}
 {\mathbf v}_i|_{\partial \Omega}={\mathbf 0}, \ \ \int_{\Omega} {\mathbf u}_j\cdot{\mathbf v}_i =0, \quad
\text{for any  $i, j=1, \cdots, k$}.
\end{equation}
It then follows from the Rayleigh-Ritz inequality (cf. \cite{KS97}) that
\begin{equation}\label{Eq5}
\sigma_{k+1}\leq \frac{\int_{\Omega} \bigl\{-\Delta {\mathbf v}_i\cdot{\mathbf v}_i +\alpha (\text{div}( {\mathbf v}_i))^2\bigl\}}{\int_{\Omega} |{\mathbf v}_i|^2}.
\end{equation}
From the definition of ${\mathbf v}_i$, we derive
$$
\aligned
 \Delta {\mathbf v}_i=&\Delta (f{\mathbf u}_i)-\sum_{j=1}^k a_{ij} \Delta{\mathbf u}_j\\
=&f\Delta{\mathbf u}_i+2{\rm   grad}f\cdot{\rm   grad}({\mathbf u}_{i})+\Delta f{\mathbf u}_i
-\sum_{j=1}^k a_{ij} \Delta {\mathbf u}_j\\
=&f\biggl(-\sigma_i{\mathbf u}_i-\alpha {\rm   grad}(\text{div}({\mathbf u}_i))\biggl)
+2{\rm   grad}f\cdot{\rm   grad}({\mathbf u}_{i})+\Delta f{\mathbf u}_i\\
&-\sum_{j=1}^k a_{ij} \biggl(-\sigma_j{\mathbf u}_j-\alpha {\rm   grad}(\text{div}({\mathbf u}_j))\biggl)\\
=&-\sigma_if{\mathbf u}_i +\sum_{j=1}^k a_{ij} \sigma_j{\mathbf u}_j
+2{\rm   grad}f\cdot{\rm   grad}({\mathbf u}_{i})+\Delta f{\mathbf u}_i\\
&-\alpha f{\rm   grad}(\text{div}({\mathbf u}_i))
+\alpha \sum_{j=1}^k a_{ij} {\rm   grad}(\text{div}({\mathbf u}_j)).
\endaligned
$$
Therefore, we have
\begin{equation}\label{Eq6}
\aligned
\int_{\Omega} -\Delta {\mathbf v}_i\cdot {\mathbf v}_i
 =&\sigma_i\parallel{\mathbf v}_i\parallel^2-\int_{\Omega}(2{\rm   grad}f\cdot{\rm   grad}({\mathbf u}_{i})+\Delta f{\mathbf u}_i)\cdot {\mathbf v}_{i}\\
& +\alpha\biggl(\int_{\Omega} f{\rm   grad}(\text{div}({\mathbf u}_i))\cdot{\mathbf v}_{i} - \sum_{j=1}^ka_{ij}
\int_{\Omega}  {\rm   grad}(\text{div}({\mathbf u}_j))\cdot {\mathbf v}_i \biggl).
\endaligned
\end{equation}
 From Stokes' theorem, we infer
 $$
 \aligned
 &\int_{\Omega} f{\rm   grad}(\text{div}({\mathbf u}_i))\cdot{\mathbf v}_{i} - \sum_{j=1}^ka_{ij}
\int_{\Omega}  {\rm   grad}(\text{div}({\mathbf u}_j))\cdot {\mathbf v}_i \\
&=-\int_{\Omega}(\text{div} ({\mathbf v}_i ))^2
+\int_{\Omega}\biggl (\text{div}({\mathbf v}_i){\rm   grad}f\cdot{\mathbf u}_i
 -\text{div}({\mathbf u}_i){\rm   grad}f\cdot{\mathbf v}_i\biggl)\\
 &=-\int_{\Omega}(\text{div} ({\mathbf v}_i ))^2
-\int_{\Omega}\biggl ({\rm   grad}({\rm   grad}f\cdot{\mathbf u}_i)
 +\text{div}({\mathbf u}_i){\rm   grad}f\biggl)\cdot{\mathbf v}_i.\\
\endaligned
 $$
From (\ref{Eq5}) and (\ref{Eq6}), we have
\begin{equation}\label{Eq7}
 \aligned
 (\sigma_{k+1}-\sigma_i)\parallel{\mathbf v}_i\parallel^2
&\le -\int_{\Omega}\biggl\{2{\rm   grad}f\cdot{\rm   grad}({\mathbf u}_{i})+\Delta f{\mathbf u}_i\\
&\qquad\qquad+\alpha\biggl ({\rm   grad}({\rm   grad}f\cdot{\mathbf u}_i)
 +\text{div}({\mathbf u}_i){\rm   grad}f\biggl)\biggl\}\cdot {\mathbf v}_{i}.
 \endaligned
\end{equation}
Define
\begin{equation}\label{Eq8}
b_{ij}=\int_{\Omega} \Big({\rm   grad}f\cdot{\rm   grad}({\mathbf u}_{i})+\dfrac12\Delta f{\mathbf u}_i\Big)\cdot {\mathbf u}_j=-b_{ji}.
\end{equation}
From (\ref{Eq2}), we derive
$$
\begin{aligned}
 b_{ij}&=\int_{\Omega} \Big({\rm   grad}f\cdot{\rm   grad}({\mathbf u}_{i})+\dfrac12\Delta f{\mathbf u}_i\Big)\cdot {\mathbf u}_j\\
&=\frac12\into\Big(\Delta(f\um_i)-f\Delta \um_i\Big)\um_j\\
&=\frac12\into f\um_i\Delta \um_j+(-\Delta \um_i)f\um_j\\
&=-\frac12\into f\um_i\Big(\sigma_j\um_j+\alpha {\rm   grad}(\text{div} \um_j)\Big)
+\frac12\into f\um_j\Big(\sigma_i\um_i+\alpha {\rm   grad}(\text{div} \um_i)\Big)\\
&=\frac12(\sigma_i-\sigma_j)a_{ij}+\frac12\alpha \int_{\Omega}\biggl(  {\rm   grad}f\cdot {\mathbf u}_i\text{div}({\mathbf u}_j)
 -\text{div}({\mathbf u}_i){\rm   grad}f\cdot {\mathbf u}_j\biggl).
\end{aligned}
$$
Hence, we have
\begin{equation}\label{Eq9}
2b_{ij}=(\sigma_i -\sigma_j)a_{ij}
 +\alpha \int_{\Omega}\biggl(  {\rm   grad}f\cdot {\mathbf u}_i\text{div}({\mathbf u}_j)
 -\text{div}({\mathbf u}_i){\rm   grad}f\cdot {\mathbf u}_j\biggl).
\end{equation}
By a simple calculation, we have, from (\ref{Eq3}) and (\ref{Eq8}),
\begin{equation}\label{Eq10}
 \aligned
 &\int_{\Omega}\bigl(2{\rm   grad}f\cdot{\rm   grad}({\mathbf u}_{i})+\Delta f{\mathbf u}_i\bigl)\cdot {\mathbf v}_{i}
 =-\int_{\Omega}|{\rm   grad}f|^2 |{\mathbf u}_{i}|^2-2\sum_{j=1}^k a_{ij}b_{ij},
 \endaligned
\end{equation}
\begin{equation}\label{Eq11}
 \aligned
&\int_{\Omega}\bigl ({\rm   grad}({\rm   grad}f\cdot{\mathbf u}_i)
 +\text{div}({\mathbf u}_i){\rm   grad}f\bigl)\cdot {\mathbf v}_{i}\\
 &=\sum_{j=1}^k a_{ij}\int_{\Omega}\bigl(  {\rm   grad}f\cdot {\mathbf u}_i\text{div}({\mathbf u}_j)
 -\text{div}({\mathbf u}_i){\rm   grad}f\cdot {\mathbf u}_j\bigl)
 -\int_{\Omega}|{\rm   grad}f\cdot{\mathbf u}_i|^2.
 \endaligned
\end{equation}
 Putting
$$
w_i=-\int_{\Omega}\biggl\{2{\rm   grad}f\cdot{\rm   grad}({\mathbf u}_{i})+\Delta f{\mathbf u}_i
+\alpha\biggl ({\rm   grad}({\rm   grad}f\cdot{\mathbf u}_i)
 +\text{div}({\mathbf u}_i){\rm   grad}f\biggl)\biggl\}\cdot {\mathbf v}_{i},
 $$
we derive from (\ref{Eq9})-(\ref{Eq11}) that
\begin{equation}\label{Eq12}
w_i=\int_{\Omega}|{\rm   grad}f|^2 |{\mathbf u}_{i}|^2+\sum_{j=1}^k (\sigma_i-\sigma_j)a_{ij}^2
+\alpha\int_{\Omega} |{\rm   grad}f\cdot{\mathbf u}_i|^2.
\end{equation}
We infer, from (\ref{Eq7})  and (\ref{Eq12}),
\begin{equation}\label{Eq13}
(\sigma_{k+1}-\sigma_i)\parallel{\mathbf v}_i\parallel^2
\le w_i.
\end{equation}

On the other hand, from (\ref{Eq4}), (\ref{Eq9}) and the inequality of Cauchy-Schwarz, we have
$$
 \aligned
  w_i^2=&\biggl(-\int_{\Omega}\biggl\{2{\rm  grad}f\cdot{\rm   grad}({\mathbf u}_{i})
  +\Delta f{\mathbf u}_i\\
&\qquad+\alpha\bigl\{{\rm   grad}({\rm   grad}f\cdot{\mathbf u}_i)
 +\text{div}({\mathbf u}_i){\rm   grad}f\bigl\}-\sum_{j=1}^k(\sigma_i-\sigma_j)a_{ij}{\mathbf u}_j
 \biggl\}\cdot {\mathbf v}_{i}\biggl)^2\\
 \leq &\|{\mathbf v_i}\|^2\bigl\|2{\rm   grad}f\cdot{\rm   grad}({\mathbf u}_{i})+\Delta f{\mathbf u}_i\\
&\qquad+\alpha\bigl \{{\rm grad}({\rm   grad}f\cdot{\mathbf u}_i)
 +\text{div}({\mathbf u}_i){\rm   grad}f\bigl\}-\sum_{j=1}^k(\sigma_i-\sigma_j)a_{ij}{\mathbf u}_j\bigl\|^2\\
 =&\|{\mathbf v_i}\|^2\biggl\{\bigl\|2{\rm   grad}f\cdot{\rm   grad}({\mathbf u}_{i})+\Delta f{\mathbf u}_i
 +\alpha\bigl\{{\rm   grad}({\rm   grad}f\cdot{\mathbf u}_i)
 +\text{div}({\mathbf u}_i){\rm   grad}f\bigl\}\bigl\|^2\\
 &\qquad-\sum_{j=1}^k(\sigma_i-\sigma_j)^2a_{ij}^2\biggl\}.
 \endaligned
$$
Hence, we infer from (\ref{Eq13})
$$
\aligned
(\sigma_{k+1}-\sigma_i)^2 w_i^2
\leq& (\sigma_{k+1}-\sigma_i)w_i\biggl\{\bigl\|2{\rm   grad}f\cdot{\rm   grad}({\mathbf u}_{i})+\Delta f{\mathbf u}_i\\
&+\alpha\Big({\rm   grad}({\rm   grad}f\cdot{\mathbf u}_i)
 +\text{div}({\mathbf u}_i){\rm   grad}f\Big)\bigl\|^2-\sum_{j=1}^k(\sigma_i-\sigma_j)^2a_{ij}^2\biggl\},
 \endaligned
 $$
\begin{equation}\label{Eq14}
\aligned
(\sigma_{k+1}-\sigma_i)^2 w_i
&\leq(\sigma_{k+1}-\sigma_i)
\biggl\{\bigl\|2{\rm   grad}f\cdot{\rm   grad}({\mathbf u}_{i})+\Delta f{\mathbf u}_i\\
&+\alpha\Big({\rm   grad}({\rm   grad}f\cdot{\mathbf u}_i)
 +\text{div}({\mathbf u}_i){\rm   grad}f\Big)\bigl\|^2-\sum_{j=1}^k(\sigma_i-\sigma_j)^2a_{ij}^2\biggl\}.
 \endaligned
\end{equation}
Taking sum on $i$ from $1$ to $k$ for (\ref{Eq14}), we have, from (\ref{Eq12}) and $a_{ij}=a_{ji}$,
$$
\aligned
&\sum_{i=1}^{k}(\sigma_{k+1}-\sigma_i)^2\biggl \{\int_{\Omega}|{\rm   grad}f|^2|{\mathbf u}_i|^2
+\alpha\int_{\Omega}|{\rm   grad}f\cdot{\mathbf u}_i|^2\biggl\}\\
&\leq\sum_{i=1}^k(\sigma_{k+1}-\sigma_i)\bigl\|2{\rm   grad}f\cdot{\rm   grad}({\mathbf u}_i)
+\Delta f {\mathbf u}_i
+\alpha\bigl\{{\rm grad}({\rm grad}f\cdot{\mathbf u}_i)
+\text{div}({\mathbf u}_i){\rm grad}f\bigl\}\bigl\|^2.
\endaligned
$$
The first inequality of Lemma \ref{Mlem} is proved.

For any constant $B> 0$, we infer, from (\ref{Eq4}), (\ref{Eq10}) and (\ref{Eq13}),
\begin{equation}\label{Eq15}
 \aligned &
 (\sigma_{k+1}-\sigma_i)^2\bigg(\int_{\Omega}|{\rm   grad}f|^2 |{\mathbf u}_{i}|^2+2\sum_{j=1}^k a_{ij}b_{ij}\bigg) \\
 &= (\sigma_{k+1}-\sigma_i)^2\biggl\{-2\int_{\Omega}\bigl({\rm   grad}f\cdot{\rm   grad}({\mathbf u}_{i})+\dfrac12\Delta f{\mathbf u}_i-\sum_{j=1}^kb_{ij}{\mathbf u}_j\bigl)\cdot {\mathbf v}_{i}\biggl\}\\
& \le (\sigma_{k+1}-\sigma_i)^3B \| {\mathbf v}_i \|^2 +\frac {\sigma_{k+1}-\sigma_i}{B}
 \biggl(\bigl\| {\rm   grad}f\cdot{\rm   grad}({\mathbf u}_{i})+\dfrac12\Delta f{\mathbf u}_i\bigl\|^2
 -\sum_{j=1}^k b_{ij}^2 \biggl)\\
&\le (\sigma_{k+1}-\sigma_i)^2B\biggl(\int_{\Omega}|{\rm   grad}f|^2 |{\mathbf u}_{i}|^2+\sum_{j=1}^k (\sigma_i-\sigma_j)a_{ij}^2
+\alpha\int_{\Omega} |{\rm   grad}f\cdot{\mathbf u}_i|^2\biggl)\\
&\qquad+\frac {\sigma_{k+1}-\sigma_i}{B}
 \biggl(\bigl\| {\rm   grad}f\cdot{\rm   grad}({\mathbf u}_{i})+\dfrac12\Delta f{\mathbf u}_i\bigl\|^2
 -\sum_{j=1}^k b_{ij}^2 \biggl).\endaligned
\end{equation}
  Taking sum on $i$ from 1 to  $k$ for (\ref{Eq15}), we obtain
 $$
 \aligned
 &
 \sum_{i=1}^k(\sigma_{k+1}-\sigma_i)^2
 \bigl(\int_{\Omega}|{\rm   grad}f|^2 |{\mathbf u}_{i}|^2+2\sum_{j=1}^k a_{ij}b_{ij}\bigl) \\
&\le \sum_{i=1}^k(\sigma_{k+1}-\sigma_i)^2B\biggl(\int_{\Omega}|{\rm   grad}f|^2 |{\mathbf u}_{i}|^2+\sum_{j=1}^k (\sigma_i-\sigma_j)a_{ij}^2
+\alpha\int_{\Omega} |{\rm   grad}f\cdot{\mathbf u}_i|^2\biggl)\\
&\qquad+\sum_{i=1}^k\frac {\sigma_{k+1}-\sigma_i}{B}
 \biggl(\bigl\| {\rm   grad}f\cdot{\rm   grad}({\mathbf u}_{i})+\dfrac12\Delta f{\mathbf u}_i\bigl\|^2
 -\sum_{j=1}^k b_{ij}^2 \biggl).\\
 \endaligned
  $$
Since $a_{ij}$ is symmetric and $b_{ij}$ is anti-symmetric, we have
 $$
\begin{aligned}
 2 \sum_{i,j=1}^k(\sigma_{k+1}- \sigma_i )^2a_{ij}b_{ij} &=-2\sum_{i,j=1}^k(\sigma_{k+1}- \sigma_i )(\sigma_{i}- \sigma_j )a_{ij}b_{ij},\\
\sum_{i,j=1}^k ( \sigma_{k+1} -\sigma_i) ^2 (\sigma_i-\sigma_j)a_{ij}^2
 &=-\sum_{i,j=1}^k(\sigma_{k+1} -\sigma_i) (\sigma_i-\sigma_j)^2a_{ij}^2.
\end{aligned}
$$
Therefore, we infer
 $$
 \aligned
 &
 \sum_{i=1}^k(\sigma_{k+1}-\sigma_i)^2
 \int_{\Omega}|{\rm   grad}f|^2 |{\mathbf u}_{i}|^2\\
&\le \sum_{i=1}^k(\sigma_{k+1}-\sigma_i)^2B\biggl(\int_{\Omega}|{\rm   grad}f|^2 |{\mathbf u}_{i}|^2
+\alpha\int_{\Omega} |{\rm   grad}f\cdot{\mathbf u}_i|^2\biggl)\\
&\qquad+\sum_{i=1}^k\frac {\sigma_{k+1}-\sigma_i}{B}
 \biggl(\bigl\| {\rm   grad}f\cdot{\rm   grad}({\mathbf u}_{i})+\dfrac12\Delta f{\mathbf u}_i\bigl\|^2  \biggl).\\
 \endaligned
$$
 This finishes the proof of Lemma \ref{Mlem}.
\end{proof}

Next, we shall give a proof of Theorem \ref{Mthm}.\\

\noindent
{\it Proof of Theorem \ref{Mthm}.}
For the standard Euclidean coordinate system $(x^1, x^2, \cdots, x^n)$ in $\mathbb R^n$, we have, for any $1\leq \beta\leq n$,
$$
{\rm grad}(x^{\beta})={\mathbf e}_{\beta},
$$
where  ${\mathbf e}_1=(1, 0, \cdots, 0), {\mathbf e}_2=(0, 1, \cdots, 0), {\mathbf e}_n=(0, 0, \cdots, 1)$. \par

Taking $f=x^{\beta}$  in (2.1) and making sum on $\beta$ from 1 to  $n$ for the resulted inequality,
we obtain
$$
\aligned
&\sum_{i=1}^{k}(\sigma_{k+1}-\sigma_i)^2\biggl \{(1-B)\int\sum_{\beta=1}^n|{\rm   grad}(x^{\beta})|^2|{\mathbf u}_i|^2
-B\alpha\int\sum_{\beta=1}^n\bigl|{\rm   grad}(x^{\beta})\cdot{\mathbf u}_i\bigl|^2\biggl\}\\
&\leq\dfrac1B\sum_{i=1}^k(\sigma_{k+1}-\sigma_i)\sum_{\beta=1}^n\|{\rm   grad}(x^{\beta})\cdot{\rm   grad}({\mathbf u}_i)
+\dfrac12\Delta x^{\beta} {\mathbf u}_i
\|^2.
\endaligned
$$
A straightforward calculation yields
$$
\sum_{\beta=1}^n|{\rm   grad}(x^{\beta})|^2=n,
$$
$$
\sum_{\beta=1}^n\bigl({\rm   grad}(x^{\beta})\cdot{\mathbf u}_i\bigl)^2=|{\mathbf u}_i|^2,
$$
$$
\sum_{\beta=1}^n|{\rm   grad}(x^{\beta})\cdot{\rm   grad}({\mathbf u}_i)
+\dfrac12\Delta x^{\beta} {\mathbf u}_i|^2
=\sum_{\beta=1}^n|{\mathbf e}_\beta\cdot{\rm   grad}({\mathbf u}_i)|^2.\\
$$
From Stokes' formula, we have
$$
\into\sum_{\beta=1}^n|{\mathbf e}_\beta\cdot{\rm   grad}({\mathbf u}_i)|^2
=\sigma_i-\alpha\|\text{div}\um_i\|^2.
$$
Therefore, we infer
$$
\sum_{i=1}^{k}(\sigma_{k+1}-\sigma_i)^2\Big(n-B(n+\alpha)\Big)
\leq\dfrac1B\sum_{i=1}^k(\sigma_{k+1}-\sigma_i)(\sigma_i-\alpha\|\text{div}({\mathbf u}_i)\|^2).
$$
Putting
$$
B=\dfrac{\sqrt{\sum_{i=1}^k(\sigma_{k+1}-\sigma_i)(\sigma_i-\alpha\|\text{div}({\mathbf u}_i)\|^2)}}{\sqrt{(n+\alpha)\sum_{i=1}^{k}(\sigma_{k+1}-\sigma_i)^2}},
$$
we derive
\begin{equation}\label{Eq16}
\sum_{i=1}^{k}(\sigma_{k+1}-\sigma_i)^2
\leq\dfrac{4(n+\alpha)}{n^2}\sum_{i=1}^k(\sigma_{k+1}-\sigma_i)(\sigma_i-\alpha\|\text{div}({\mathbf u}_i)\|^2).
\end{equation}
Since $\alpha\geq 0$, we have
\begin{equation}\label{Y11}
\aligned
&\sum_{i=1}^{k}(\sigma_{k+1}-\sigma_i)^2
\leq\dfrac{4(n+\alpha)}{n^2}\sum_{i=1}^k(\sigma_{k+1}-\sigma_i)\sigma_i.
\endaligned
\end{equation}

On the other hand, taking $f=x^{\beta}$ in the first inequality of Lemma \ref{Mlem}, $\beta =1, 2, \cdots, n$,
we infer
$$
\aligned
&\sum_{i=1}^{k}(\sigma_{k+1}-\sigma_i)^2\biggl (1
+\alpha\int_{\Omega}({\mathbf e}_{\beta}\cdot{\mathbf u}_i)^2\biggl)\\
&\leq\sum_{i=1}^k(\sigma_{k+1}-\sigma_i)
\bigl\|2{\mathbf e}_{\beta}\cdot{\rm   grad}({\mathbf u}_i)
+\alpha\Big({\rm grad}({\mathbf e}_{\beta}\cdot{\mathbf u}_i)
+\text{div}({\mathbf u}_i){\mathbf e}_{\beta}\Big)\bigl\|^2.
\endaligned
$$
Taking sum on $\beta$ from $1$ to $n$, we have
$$
\aligned
&(n+\alpha)\sum_{i=1}^{k}(\sigma_{k+1}-\sigma_i)^2\\
&\leq\sum_{i=1}^k(\sigma_{k+1}-\sigma_i)
\sum_{\beta=1}^n\bigl\|2{\mathbf e}_{\beta}\cdot{\rm   grad}({\mathbf u}_i)
+\alpha\Big({\rm grad}({\mathbf e}_{\beta}\cdot{\mathbf u}_i)
+\text{div}({\mathbf u}_i){\mathbf e}_{\beta}\Big)\bigl\|^2.
\endaligned
$$
By a simple and direct computation, we infer
$$
\aligned
&\sum_{\beta=1}^n\bigl\|2{\mathbf e}_{\beta}\cdot{\rm   grad}({\mathbf u}_i)
+\alpha\Big({\rm grad}({\mathbf e}_{\beta}\cdot{\mathbf u}_i)
+\text{div}({\mathbf u}_i){\mathbf e}_{\beta}\Big)\bigl\|^2\\
&=(4+\alpha^2)\sigma_i
-\alpha\Big(\alpha^2-(n+2)\alpha-4\Big)\bigl\|{\rm div}({\mathbf u}_i)\bigl\|^2.
\endaligned
$$
Hence, we have
\begin{equation}\label{Eq17}
\aligned
\sum_{i=1}^{k}&(\sigma_{k+1}-\sigma_i)^2\\
&\leq\sum_{i=1}^k(\sigma_{k+1}-\sigma_i)\biggl(\dfrac{4+\alpha^2}{n+\alpha}\sigma_i
-\alpha\dfrac{\alpha^2-(n+2)\alpha-4}{n+\alpha}
\bigl\|{\rm div}({\mathbf u}_i)\bigl\|^2\biggl).
\endaligned
\end{equation}
For $\alpha\geq \dfrac{n+2+\sqrt{(n+2)^2+16}}2$, we have
$\alpha^2-(n+2)\alpha-4\geq 0$. Hence,
$$
\aligned
\sum_{i=1}^{k}(\sigma_{k+1}-\sigma_i)^2
\leq\dfrac{4+\alpha^2}{n+\alpha}\sum_{i=1}^k(\sigma_{k+1}-\sigma_i)\sigma_i.
\endaligned
$$
For $0\leq \alpha< \dfrac{n+2+\sqrt{(n+2)^2+16}}2$, we have
$\alpha^2-(n+2)\alpha-4<0$. In this case, from
$L=\dfrac{\{4+(n+2)\alpha-\alpha^2\}n^2}{4(n+\alpha)^2}>0$,  (\ref{Eq16}) and (\ref{Eq17}), we have
$$
\aligned
\sum_{i=1}^{k}(\sigma_{k+1}-\sigma_i)^2
\leq\dfrac{8+(n+2)\alpha}{(n+\alpha)(1+L)}\sum_{i=1}^k(\sigma_{k+1}-\sigma_i)\sigma_i.
\endaligned
$$
Thus, we derive, from the definition of $A(n,\alpha)$,
\begin{equation}\label{Eq18}
\aligned
\sum_{i=1}^{k}(\sigma_{k+1}-\sigma_i)^2
\leq\dfrac{A(n,\alpha)}{n+\alpha}\sum_{i=1}^k(\sigma_{k+1}-\sigma_i)\sigma_i.
\endaligned
\end{equation}
Furthermore, from (\ref{Y11}) and (\ref{Eq18}), we infer
$$
\aligned
\sum_{i=1}^{k}(\sigma_{k+1}-\sigma_i)^2
\leq\min\biggl\{\dfrac{4(n+\alpha)}{n^2}, \dfrac{A(n,\alpha)}{n+\alpha}\biggl\}\sum_{i=1}^k(\sigma_{k+1}-\sigma_i)\sigma_i.
\endaligned
$$
This completes the proof of Theorem 1.1.$\hfill\square$

 \begin{rem}The inequality (\ref{mineq}) implies Cheng-Yang's inequality (\ref{CY}). In order to see this, we need an elementary algebraic inequality.

\begin{lem} \label{Alglem}
Let $\{a_i\}_{i=1}^k$ and
$\{b_i\}_{i=1}^k$  be two  sequences of
non-negative real numbers with $\{a_i\}_{i=1}^k$ decreasing  and $\{b_i\}_{i=1}^k$ increasing. Then  for any fixed $s\geq 1$, we have
\be \label{Alg}
 \left(\sum_{i=1}^k a_i^s\right)\left(\sum_{i=1}^k a_i^2 b_i
\right)\leq \left(\sum_{i=1}^k a_i^{s+1} \right)\left(\sum_{i=1}^k
a_i b_i \right).
\en
\end{lem}
\begin{proof} When $k=1$, (\ref{Alg}) holds
trivially. Suppose that (\ref{Alg}) holds when  $k=m$, that is,
 \be \label{Assum}
 \left(\sum_{i=1}^m a_i^s\right)\left(\sum_{i=1}^m a_i^2 b_i
\right)\leq \left(\sum_{i=1}^m a_i^{s+1} \right)\left(\sum_{i=1}^m
a_i b_i\right).
\en
Then when $k=m+1$, we have by using (\ref{Assum}) and the hypothesis on $\{a_i\}_{i=1}^k$  and $\{b_i\}_{i=1}^k$  that
\begin{equation*}
\begin{aligned}
&\sum_{i=1}^{m+1} a_i^{s+1} \sum_{i=1}^{m+1}
a_i b_i-\sum_{i=1}^{m+1} a_i^s\sum_{i=1}^{m+1} a_i^2 b_i\\
&=\sum_{i=1}^m a_i^{s+1} \sum_{i=1}^m
a_i b_i -\sum_{i=1}^m a_i^s\sum_{i=1}^m a_i^2 b_i
+ a_{m+1}^{s+1}\sum_{i=1}^m
a_ib_{i}\\
&\quad-a_{m+1}^2b_{m+1}\sum_{i=1}^m
a_i^s + a_{m+1}b_{m+1}\sum_{i=1}^m
a_i^{s+1}-a_{m+1}^s\sum_{i=1}^m
a_i^2b_i\\
&\geq a_{m+1}^{s+1}\sum_{i=1}^m
a_ib_{i}-a_{m+1}^2b_{m+1}\sum_{i=1}^m
a_i^s + a_{m+1}b_{m+1}\sum_{i=1}^m
a_i^{s+1}-a_{m+1}^s\sum_{i=1}^m
a_i^2b_i\\
&=\sum_{i=1}^m a_{m+1}a_i (b_{m+1}a_i^{s-1}-b_ia_{m+1}^{s-1})(a_i-a_{m+1})\geq 0.
\end{aligned}
\end{equation*}

Thus (\ref{Alg}) holds by induction.    $\hfill\square$

Now let us get (\ref{CY}) by using (\ref{mineq}). Multiplying (\ref{mineq}) by $\left(\sum_{i=1}^k(\si_{k+1}-\si_i)\right)^2$, we get
\begin{equation}\label{Eq19}
\begin{aligned}
&\left(\sum_{i=1}^k(\si_{k+1}-\si_i)\right)^2 \left(\sum_{i=1}^k (\si_{k+1}-\si_i)^2\right)\\
&\leq\fr{4(n+\alpha)}{n^2}\left(\sum_{i=1}^k(\si_{k+1}-\si_i)\right)^2\left(\sum_{i=1}^k (\si_{k+1}-\si_i)\si_i\right).
\end{aligned}
\end{equation}

Taking $s=2, \ a_i=(\si_{k+1}-\si_i)^{1/2}, \ b_i=\si_i$ in (\ref{Alg}), we get
\begin{equation}\label{Eq20}
\begin{aligned}
&\Big(\sum_{i=1}^k(\si_{k+1}-\si_i)\Big)\Big(\sum_{i=1}^k (\si_{k+1}-\si_i)\si_i\Big)\\
&\leq  \Big(\sum_{i=1}^k(\si_{k+1}-\si_i)^{3/2}\Big)\Big(\sum_{i=1}^k (\si_{k+1}-\si_i)^{1/2}\si_i\Big).
\end{aligned}
\end{equation}

Taking $s=3, \ a_i=(\si_{k+1}-\si_i)^{1/2}, \ b_i\equiv 1$ in (\ref{Alg}), we have
\be \label{Eq21}
& &
\left(\sum_{i=1}^k(\si_{k+1}-\si_i)^{3/2}\right)\left(\sum_{i=1}^k (\si_{k+1}-\si_i)\right)
\\ \no
&\leq & \left(\sum_{i=1}^k(\si_{k+1}-\si_i)^2\right)\left(\sum_{i=1}^k (\si_{k+1}-\si_i)^{1/2}\right).
\en
It is then easy to obtain (\ref{CY}) from (\ref{Eq19})-(\ref{Eq21}).
\end{proof}
\end{rem}

\section{Proof of Theorems \ref{EPbi1}-\ref{EPbu2} }
\setcounter{equation}{0}
In this section, we will prove theorems \ref{EPbi1}-\ref{EPbu2}. Before doing this, let us  recall the Reilly formula. Let $M$ be $n$-dimensional compact manifold $M$ with boundary $\pa M$. We will often write $\lan, \ran$ the Riemannian metric on $M$ as well as that induced on $\pa M$. Let $\na$ and $\D $ be the connection  and the Laplacian on $M$,
respectively. Let $\nu$ be the unit outward normal vector of $\pa M$. The shape operator of $\pa M$ is given by $S(X)=\na_X \nu$ and the second fundamental form of $\pa M$ is defined as $II(X, Y)=\lan S(X), Y\ran$, here $X, Y\in T \pa M$. The eigenvalues of $S$ are called the principal curvatures of $\pa M$ and the mean curvature $H$ of $\pa M$ is given by $H=\fr 1{n-1} {\rm tr\ } S$, here ${\rm tr\ } S$ denotes the trace of $S$.
For a smooth function $f$ defined on an $n$-dimensional compact manifold $M$ with boundary $\pa M$, the following identity holds if
$h=\left.\fr{\pa f}{\pa \nu}\right|_{\pa M}$, $z=f|_{\pa M}$ and ${\rm Ric}$ denotes the Ricci tencor of $M$ (Cf. \cite{Reilly77}, p. 46):
\be \label{Reilly}
& & \int_M \left((\D f)^2-|\na^2 f|^2-{\rm Ric}(\na f, \na f)\right)\\ \no
&=& \int_{\pa M}\left( ((n-1)Hh+2\overline{\D}z)h + II(\overline{\na}z, \overline{\na}z)\right).
\en
Here $\na^2 f$ is the Hessian of $f$; $\ov{\D}$ and $\ov{\na}$  represent the Laplacian and the gradient on $\pa M$ with respect
to the induced metric on $\pa M$,  respectively.

\vs
\noindent
{\it Proof of Theorem \ref{EPbi1}.} Let $u$ be an eigenfunction of the problem (\ref{EP2}) corresponding to the first eigenvalue $\Gamma_1$. That is,
  \be\no  \Delta^2 u= \Gamma_1  u   \ \ {\rm in \ \ } M, \ \ u= \fr{\pa u}{\pa \nu}=0 \ \ {\rm on \ \ } \pa M.
\en
 Then we have
 \be \label{Ebi1}
 \Gamma_1=\fr{\int_M (\D u)^2}{\int_M u^2}.
 \en
 Introducing $u$ into Reilly's formula, it follows that
\begin{equation}\label{Ebi2}
\int_M \left((\D u)^2-|\na^2 u|^2\right)=\int_M{\rm Ric}(\na u, \na u)
\geq (n-1)\int_M |\na u|^2.
\end{equation}
 From the Schwarz inequality, we have
\be \label{Sch}
|\na^2 u|^2\geq \fr 1n (\D u)^2
\en
with equality holding if and only if
\be\no
\na^2 u =\fr{\D u}n \lan , \ran.
\en
Thus we have from (\ref{Ebi2}) and (\ref{Sch}) that
\be \label{Ebi3}
\int_M (\D u)^2 \geq n\int_M |\na u|^2.
\en
Since $u$ is not a zero function which vanishes on $\pa M$, we have from the Poincar\'e inequality that
\be \label{Ebi4}
\int_M |\na u|^2\geq \lambda_1 \int_M u^2
\en
with equality holding if and only if $u$ is a first eigenfunction of the Dirichlet Laplacian of $M$.

Combining (\ref{Ebi1}), (\ref{Ebi3}) and (\ref{Ebi4}), we get $\Gamma_1\geq n\lambda_1$. Let us show that the case $\Gamma_1=n\lambda_1$ will not occur.
Indeed, if $\Gamma_1=n\lambda_1$, then we must have
\be \label{Ebi5}
|\na^2 u|^2= \fr 1n (\D u)^2, \ \ \ {\rm Ric}(\na u, \na u)=(n-1)|\na u|^2, \ \D u=-\lam_1 u.
\en
Hence
\be
\Gamma_1 u= \D (\D u)=\D (-\lambda_1 u)= \lambda_1^2 u \ \ \ {\rm in }\ \ \ M
\en
which implies that $\lambda_1=n$. Consequently,
we get from the Bochner formula that
\be\no
\fr 12 \D (|\na u|^2+u^2)&=&|\na^2u|^2+\lan \na u, \na(\D u)\ran +{\rm Ric}(\na u, \na u) +|\na u|^2+ u\D u\\ \no
&=& \fr{(\D u)^2}n -n\lan \na u, \na u\ran + (n-1)|\na u|^2+|\na u|^2-n u^2\\ \no
&=&0.
\en
Since $(|\na u|^2+u^2)|_{\pa M}=0,$ we conclude from the maximum principle that $|\na u|^2+u^2=0$ on $M$. This is  a contradiction and completes the proof of Theorem 1.5. $\hfill\square$

\vs
\noindent
{\it Proof of Theorem \ref{EPbu1}.} Let $w$ be an eigenfunction of the problem (1.8) corresponding to the first eigenvalue $\Lambda_1$. That is,
  \be \label{Ebu1}
   \Delta^2 w= -\Lambda_1\D  w   \ \ {\rm in \ \ } M,\ \ w= \fr{\pa w}{\pa \nu}=0 \ \ {\rm on \ \ } \pa M.
\en
 Then we have
 \be \label{Ebu2}
 \Lambda_1=\fr{\int_M (\D w)^2}{\int_M |\na w|^2}.
 \en
 As in the proof of Theorem \ref{EPbi1}, we have by introducing $w$ into Reilly's formula that
\be \label{Ebu3}
 \int_M (\D w)^2\geq  n\int_M |\na w|^2
\en
 with equality holding if and only if
\be \label{Ebu4}
\na^2 w =\fr{\D w}n \lan , \ran \ \ \ {\rm and}\ \ \ {\rm Ric}(\na w, \na w)=(n-1)|\na w|^2.
\en
Combining (3.10) and (3.11), we get $\Lambda_1\geq n$. If $\Lambda_1=n$, then   (\ref{Ebu4}) holds.  Since
\be
w|_{\pa M}= \left.\fr{\pa w}{\pa \nu}\right|_{\pa M}=0,
\en
we have
\be
\D w|_{\pa M}=\na^2 w(\nu,\nu),
\en
which, combining with $\na^2 w=\fr{\D w}n \lan, \ran$, implies that
\be
\D w|_{\pa M}=0.
\en
It then follows from the divergence theorem that
\be\no
\int_M |\na(\D w+ nw)|^2&=&-\int_M (\D w+ nw)\D (\D w+ nw)\\ \no &=&-\int_M (\D w+ nw) (\D^2 w+ n\D w)=0.
\en
Hence
\be
\D w+nw =0\ \ \ {\rm in} \ \ \ M.
\en
Consequently, we get
 \be\no
\fr 12 \D (|\na w|^2+w^2)&=&|\na^2w|^2+\lan \na w, \na(\D w)\ran +{\rm Ric}(\na w, \na w) +|\na w|^2+ w\D w\\ \no
&=& \fr{(\D w)^2}n -n\lan \na w, \na w)\ran + (n-1)|\na w|^2+|\na w|^2-n w^2\\ \no
&=&0.
\en
 We then conclude from $(|\na w|^2+w^2)|_{\pa M}=0$ and the maximum principle that $|\na w|^2+w^2=0$. This is a contradiction and so $\Lambda_1>n$. The proof of Theorem \ref{EPbu1} is completed.$\hfill\square$

\vs
\noindent
{\it Proof of Theorem \ref{EPbi2}.} Let $f$ be the eigenfunction of the problem (\ref{EP3}) corresponding to the first eigenvalue $p_1$. That is,
 \begin{equation}\label{Ebi6}
\left \{ \aligned &\Delta^2 f= p_1  f   \ \ {\rm in \ \ } M, \\
 &f= \fr{\pa^2 f}{\pa \nu^2}=0 \ \ {\rm on \ \ } \pa M.
 \endaligned \right.
 \end{equation}
Multiplying (\ref{Ebi6}) by $f$ and integrating on $M$, we have from the divergence theorem that
\be \label{Ebi7}
p_1\int_M f^2&=& \int_M f\D^2 f\\ \no &=&-\int_{M}\lan \na f, \na(\D f)\ran \\ \no &=&
\int_M (\D f)^2 -\int_{\pa M} h\D f,
\en
where $h=\left.\fr{\pa f}{\pa \nu}\right|_{\pa M}.$ Since $f|_{\pa M}=\left. \fr{\pa^2 u}{\pa \nu^2}\right|_{\pa M}=0$, we have
\be \label{Ebi8}
\D f|_{\pa M} = (n-1) H h,
\en
where $H$ is the mean curvature of $\pa M$.

Substituting (\ref{Ebi8}) into (\ref{Ebi7}), we get
\be \label{Ebi88}
p_1=\fr{\int_M (\D f)^2- (n-1)\int_{\pa M} Hh^2}{\int_M f^2}.
\en
Introducing $f$ into Reilly's formula, we have
\be \label{Ebi9}
 \int_M \left((\D f)^2-|\na^2 f|^2\right)&=& \int_M{\rm Ric}(\na f, \na f)+(n-1)\int_{\pa M}Hh^2
\\ \no &\geq& (n-1)\int_M |\na f|^2+(n-1)\int_{\pa M}Hh^2.
\en
It follows from the Schwarz inequality that
\be \label{Ebi10}
|\na^2 f|^2\geq \fr 1n (\D f)^2
\en
with equality holding if and only if
\be\no
\na^2 f =\fr{\D f}n \lan , \ran.
\en
Combining  (\ref{Ebi9}) and (\ref{Ebi10}), one gets
\be \label{Ebi11}
\int_M (\D f)^2 \geq n\int_M |\na f|^2 + n\int_{\pa M}Hh^2.
\en
Since $H\geq 0$, we have from (\ref{Ebi88}) and (\ref{Ebi11}) that
\be \label{Ebi12}
p_1\geq \fr{\int_M |\na f|^2}{\int_M f^2}.
\en
On the other hand, since $f$ is not a zero function which vanishes on $\pa M$, we have from the Poincar\'e inequality that
\be \label{Ebi13}
\int_M |\na f|^2\geq \lambda_1 \int_M f^2
\en
with equality holding if and only if $f$ is a first eigenfunction of the Dirichlet Laplacian of $M$.
Thus we conclude that $p_1\geq n\lambda_1$.  This finishes the proof of the first part of Theorem \ref{EPbi2}. Assume now that $p_1=n \lambda_1$. In this case, (\ref{Ebi13}) should take equality sign which implies that $f$ is a first eigenfunction corresponding to the the first eigenvalue $\lambda_1$ of the Dirichlet Laplacian of $M$. That is, we have
 \be
  \Delta f= -\lambda_1  f   \ \ {\rm in \ \ } M,\ \   f|_{\pa M}=0.
\en
It then follows that
 \be
  \Delta^2 f= -\lambda_1\D  f =\lambda_1^2 f \ \ {\rm in \ \ M}
 \en
 which, combining with (\ref{Ebi6}) gives $\lambda_1=n$. We then conclude from Reilly's theorem as stated before that $M$ is isometric to an $n$-dimensional unit semi-sphere. Consider now the $n$-dimensional unit semi-sphere $S^n_+(1)$ given by
 \be\no
  S^n_+(1)=\left\{(x_1,..., x_{n+1})\in {\bf R}^{n+1}\left. \right| \sum_{i=1}^{n+1} x_i^2=1,  \ x_{n+1}\geq 0\right\}
 \en
 It is easy to see that the function $x_{n+1}$ on  $S^n_+(1)$ is an eigenfunction of the problem (\ref{EP4}) corresponding to the eigenvalue $n^2$ and the first Dirichlet eigenvalue $\lam_1$ of $S^n_+(1)$ is $n$. Thus the first eigenvalue of the problem (\ref{EP4}) of  $S^n_+(1)$ is $n\lam_1$.
 This completes the proof of Theorem \ref{EPbi2}.$\hfill\square$

 \vs
\noindent
 {\it Proof of Theorem \ref{EPbu2}.} The discussion is similar to the proof of Theorem 1.7. For the sake of completeness, we include it.
 Let $g$ be the eigenfunction of the problem (\ref{EP5}) corresponding to the first eigenvalue $q_1$:
\begin{equation}\label{Ebu5}
\left \{ \aligned &\Delta^2 g= -q_1 \D g   \ \ {\rm in \ \ } M, \\
&g= \fr{\pa^2 g}{\pa \nu^2}=0 \ \ {\rm on \ \ } \pa M.
\endaligned \right.
\end{equation}

Multiplying (\ref{Ebu5}) by $g$ and integrating on $M$, we have from the divergence theorem that
\be \label{Ebu6}
p_1\int_M |\na g|^2=
\int_M (\D g)^2 -\int_{\pa M} s\D g,
\en
where $s=\left.\fr{\pa g}{\pa \nu}\right|_{\pa M}.$  Also, we have
\be \label{Ebu7}
\D g|_{\pa M} = (n-1) H s.
\en
Hence
\be
q_1=\fr{\int_M (\D g)^2- (n-1)\int_{\pa M} Hs^2}{\int_M|\na g|^2}.
\en
Introducing $g$ into Reilly's formula and using Schwarz inequality, we have
\be
 \int_M (\D g)^2\geq  n\int_M |\na g|^2+n\int_{\pa M}Hs^2.
\en
Consequently, we get $q_1\geq n$. In the case that $q_1=n$, we must have $Hs=0$ and so we know from (\ref{Ebu7}) that $\D g|_{\pa M}=0$.
Observe that $\D g$ is not a zero function on $M$ since otherwise $g$ would be identically zero by the maximum principle. Thus $\D g$ is an eigenfunction corresponding the the  eigenvalue $n$ of the Dirichlet Laplacian of $M$. It then follows from Reilly's theorem that $M$ is isometric to an $n$-dimensional Euclidean semi-sphere. One can check that the function $x_{n+1}$ given in the proof of Theorem \ref{EPbi2} is an eigenfunction of the problem (\ref{EP4}) for the $n$-dimensional unit semi-sphere corresponding to the eigenvalue $n$. This completes the proof of Theorem \ref{EPbu2}.$\hfill\square$

\begin{flushleft}
Daguang Chen\\
Department  of Mathematical Sciences\\
Tsinghua University \\
Beijing 100084, China \\
e-mail: dgchen@math.tsinghua.edu.cn
\end{flushleft}

\begin{flushleft}
Qing-Ming Cheng  \\
Department of Mathematics   \\
Faculty of Science and Engineering   \\
Saga University  \\
Saga 840-8502,  Japan \\
e-mail: cheng@ms.saga-u.ac.jp
\end{flushleft}

\begin{flushleft}
Qiaoling Wang\\
Departamento de Matem\'atica\\
Universidade de
Bras\'{\i}lia\\
Bras\'{\i}lia-DF 70910-900, Brazil\\
e-mail: wang@mat.unb.br
\end{flushleft}

\begin{flushleft}
Changyu Xia\\
Departamento de Matem\'atica\\
Universidade de
Bras\'{\i}lia\\
Bras\'{\i}lia-DF 70910-900 , Brazil\\
e-mail: xia@mat.unb.br
\end{flushleft}

\end{document}